\documentclass{article}
\usepackage{amsfonts,amsmath,amssymb,amsthm}
\usepackage{tikz}
\usepackage{fullpage}
\usepackage{hyperref}
\usepackage{multicol}
\usepackage{accents}

\newcommand{\beto}{\boldsymbol{\eta}}
\newcommand{\bzeta}{\boldsymbol{\zeta}}
\newcommand{\bxi}{\boldsymbol{\xi}}

\newcommand{\HH}{\mathcal{H}}
\newcommand{\VV}{\mathcal{V}}
\newcommand{\PiHV}{(\Pi_{\HH} \otimes \Pi_{\VV})}

\newcommand{\binomq}[2]{\binom{#1}{#2}_q}
\newtheorem{theorem}{Theorem}[section]

\newtheorem{corollary}[theorem]{Corollary}
\theoremstyle{remark}

\newtheorem{reemark}{Remark}

\begin{document}

\title{A Short Note on Markov Duality in Multi--species Higher Spin Stochastic Vertex Models}

\author{Jeffrey Kuan}

\date{}

\maketitle

\abstract{We show that the multi--species higher spin stochastic vertex model, also called the $U_q(A^{(1)}_n)$ vertex model, satisfies a duality where the indicator function has the form $\{ \eta^x_{[i,n]} \geq \xi^x_{[i,n]} \}$. In other words, for every particle in the $\xi$ configuration of species $i$ at vertex $x$, there must be a particle of species $j>i$ at vertex $x$ in the $\eta$ configuration. The proof follows by applying charge reversal to previously discovered duality functions, which also results in open boundary conditions. As a corollary, we recover the duality for the stochastic six vertex model recently found by Y. Lin.}

\section{Introduction}
In \cite{KMMO}, the authors construct stochastic $S$--matrices arising from the affine quantum group $U_q(A^{(1)}_n)$. In \cite{KuanCMP}, it is shown that the corresponding stochastic vertex model satisfies a Markov duality with its space reversal. The duality function had previously occurred as the duality function between multi--species ASEP$(q,j)$ and its space reversal \cite{KIMRN}, where the function is only non--zero if
$$
\eta_0^x + \ldots + \eta_{n-i}^x \geq \xi_i^x + \ldots + \eta_n^x 
$$
for all lattice site $x$ and species number $i \in \{1,2,\ldots,n\}$. Here, $\eta_i^x$ denotes the number of particles of species $i$ at the lattice site $x$ in a configuration denoted by $\eta$. 

On the other hand, the duality function between multi--species ASEP$(q,j)$ and itself (non space--reversed) is non--zero if
$$
\eta_i^x + \ldots + \eta_{n}^x \geq \xi_i^x + \ldots + \xi_n^x.
$$
These types are duality functions are more suitable for applications of Markov duality, so it is natural to try to find these types of dualities for the stochastic $U_q(A^{(1)}_n)$ vertex models. However, for totally asymmetric models such as the stochastic vertex models, it is necessary to apply a space reversal (see the Remark after Theorem 2.5 of \cite{KIMRN}). In this paper, we show that the simultaneous inversions in the asymmetry parameter $q$ and spectral parameter $z$ will play the same role as a space reversal. This results in a new duality function for the stochastic $U_q(A^{(1)}_n)$ vertex model with open boundary conditions. For the case of the multi--species stochastic six vertex model, it generalizes the duality found in  \cite{YL19} for the single--species stochastic six vertex model. 

We note that there are duality functions for stochastic vertex models which do not have any indicator functions \cite{CorPetCMP}, \cite{YL192}, but we will not discuss these here.

\section{Main Results}

\subsection{Stochastic $S$--matrices}
First let us define some notation. Fix positive integers $l,m$ and $n$. Let $\mathcal{H}$ and $\mathcal{V}$ denote the sets
$$
\mathcal{H} = \{ (\alpha_0,\ldots,\alpha_n) :  \alpha_0 + \ldots + \alpha_n = l \}, \quad \quad  \mathcal{V}_m = \{ (\beta_0,\ldots,\beta_n) :  \beta_0 + \ldots + \beta_n = m \}.
$$
For $i<j$, let $\alpha_{[i,j]} = \alpha_i + \ldots + \alpha_j$, and similarly for $\beta$.
The paper \cite{KMMO} introduces a two--parameter family of stochastic\footnote{In this paper, we follow the common convention among mathematical physicists that a matrix is stochastic if its \textit{columns}, rather than its rows, sum to $1$.} matrices $S(q,z)$ whose rows and columns are indexed by $\mathcal{H} \times \mathcal{V}_m$. We let $S(q,z)_{\alpha\beta}^{\gamma\delta}$ denote the matrix entry with row $(\alpha,\beta)$ and column $(\gamma,\delta)$. The matrix $S(q,z)$ satisfies the conservation property that $S(q,z)$ is only nonzero if $\alpha+\beta = \gamma+\delta$. There are some explicit formulas for the matrix entries of $S(q,z)$ \cite{BoMa16}, but we will not use them here. See the end of this paper for a few examples.

We can represent $S(q,z)$ as vertex weights. Associate to each subscript $i\in \{0,\ldots,n\}$ a color. For $i<j$, we can consider particles of color $i$ to be lighter than particles of color $j$.  In the example below, we have $l=1,m=2,\alpha=(0,1,0),\beta=(1,0,1),\gamma=(0,0,1),\delta=(1,1,0)$.

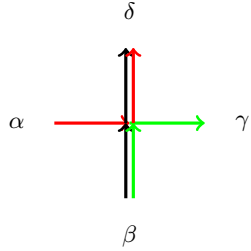
\begin{center}
\begin{tikzpicture}
\draw [->, red, very thick](7,0) -- (8,0) ;
\draw [->, green, very thick] (8,0) -- (9,0) ;
\draw [->, very thick] (7.95,-1) -- (7.95,0);
\draw [->, green, very thick] (8.05,-1) -- (8.05,0);
\draw [->, black, very thick] (7.95,0) -- (7.95,1);
\draw [->, red, very thick] (8.05,0) -- (8.05,1);

\draw (9.5,0) node {$\gamma$};
\draw (8,-1.5) node {$\beta$};
\draw (8,1.5) node {$\delta$};
\draw (6.5,0) node {$\alpha$};
\end{tikzpicture}
\end{center}

\subsection{Transfer matrices}
The stochastic $S$ matrices can be used to define an interacting particle system. Let $\bar{S}(q,z)$ denote the $\mathcal{V} \times \mathcal{H}$ by $\mathcal{H}\times \mathcal{V}$  matrix, where the entries are defined by
$$
\bar{S}(q,z)_{\beta \alpha}^{\gamma\delta} = S(q,z)_{\alpha\beta}^{\gamma\delta}.
$$
Given a positive integer $L$ and spectral parameters $\vec{z} = (z_1,\ldots,z_L)$, define the transfer matrix $\mathcal{T}_L$ to be the $\mathcal{V}_{m_1} \times \cdots \times \mathcal{V}_{m_L} \times \mathcal{H}$ by  $\mathcal{H} \times \mathcal{V}_{m_1} \times \cdots \times \mathcal{V}_{m_L}$ matrix , given by the composition
$$
\mathcal{T}_L(q,\vec{z})    = \bar{S}_{L-1,L}(q,z_L)\cdots \bar{S}_{12}(q,z_2)\bar{S}_{01}(q,z_1)
$$ 
where
$$
\bar{S}_{j,j+1}(q,z) := \mathrm{Id}^{\otimes j} \otimes \bar{S}(q,z) \otimes  \mathrm{Id}^{ \otimes L-1-j}
$$
is a $(\mathcal{V}_{m_1} \times \cdots \mathcal{V}_{m_{j+1}} \times \mathcal{H} \times \mathcal{V}_{m_{j+2}} \times \cdots \times \mathcal{V}_{m_L}) \times (\mathcal{V}_{m_1} \times \cdots \mathcal{V}_{m_j} \times \mathcal{H} \times \mathcal{V}_{m_{j+1}} \times \cdots \times \mathcal{V}_{m_L}) $ matrix.

These transfer matrices can be viewed as the transition probabilities for a discrete--time totally asymmetric particle system, either on the infinite line or on a finite lattice. We use the bold greek symbols $\boldsymbol{\eta},\boldsymbol{\xi}$ to denote elements of $\mathcal{V}_{m_1} \times \cdots \times \mathcal{V}_{m_L}$, where $\eta^x \in \mathcal{V}_{m_x}$ for $1 \leq x \leq L$. By a slight abuse of notation, we allow $L$ to equal infinity, so that $x$ can take values in $\mathbb{Z}$ or $\mathbb{Z}_{>0}$. Let 
$$
\vert \boldsymbol{\xi} \vert = \sum_x \sum_{i=1}^n \xi^x_i,
$$
and if $\vert \boldsymbol{\xi} \vert <\infty$ then we say that $\boldsymbol{\xi}$ has finitely many particles.

The stochastic matrix of transition probabilities for the discrete--time totally asymmetric particle system is then denoted $\mathcal{P}(\boldsymbol{\eta},\boldsymbol{\eta}')$, and is defined by 
$$
\mathcal{P}(\boldsymbol{\eta},\boldsymbol{\eta}') =  \mathcal{T}_L \left( (\boldsymbol{\eta},\mathbf{0}), (\mathbf{0},\boldsymbol{\eta} ' )  \right).
$$
It is not immediately obvious that $\mathcal{P}$ defines a stochastic matrix. This issue is addressed in section 4.4 of \cite{KuanCMP}. Roughly speaking, if the lattice is infinite, then with probability $1$ the particle must stop at a location finitely far from where it started, and so cannot exit at infinity. If the lattice is finite, we add auxiliary lattice sites at both ends, which have the effect of sealing off the boundaries, so that particles cannot exit or enter. 

We also let $\mathcal{P}_{\text{rev}}$ denote the space--reversed version of $\mathcal{P}$. In $\mathcal{P}$, the particles jump to the right, whereas in $\mathcal{P}_{\text{rev}}$, the particles jump to the left. We also let $\accentset{\circ}{\mathcal{P}}_{\text{rev}}$ denote the process $\mathcal{P}_{\text{rev}}$ where $l$ particles of color $n$ enter from the right boundary at every time step. Such boundary conditions had appeared previously in \cite{Bor16}, in the single--species case. See also \cite{BorWheCol} for the multi--species case. 

If the dependence on the asymmetry parameter $q$ and the spectral parameter $z$ needs to be specified, write $\mathcal{P}(q,\vec{z})$ or $\mathcal{P}_{\text{rev}}(q,\vec{z})$ or $\accentset{\circ}{\mathcal{P}}_{\text{rev}}(q,\vec{z})$

Let $D$ be the duality function 
$$
D(\xi,\eta) = \prod_x [\eta_0^x]_q^! \cdots [\eta_n^x]_q^! \prod_{i=1}^{n} \binomq{\eta_{[i,n]}^x - \xi_{[i+1,n]}^x}{\eta_i^x} \cdot  q^{-\xi_i^x \left( \eta_{[0,i-1]}^x + \sum_{z>x} 2\eta^z_{[0,i-1]}   \right) } ,
$$
where $[k]_q$ is the $q$--deformed integer
$$
[k]_q = \frac{q^k - q^{-k}}{q-q^{-1}},
$$
$[k]_q^!$ is the $q$--deformed factorial $[1]_q[2]_q \cdots [k]_q$, and
$$
\binomq{n}{k} = \frac{[n]_q^! }{[k]_q^! [n-k]_q^!}
$$
is the $q$--deformed binomial. If $n<k$ we set $\binomq{n}{k}=0$. In particular, $D(\xi,\eta)$ is only nonzero if $\eta_{[j,n]}^x \geq \xi_{[j,n]}^x$ for all $x$ and $j$. Note that when all $m_x=m$, then $D$ can also be written as 
$$
D(\xi,\eta) = \text{const} \cdot \prod_x [\eta_0^x]_q^! \cdots [\eta_n^x]_q^! \prod_{i=1}^n \binomq{\eta_{[i,n]}^x - \xi_{[i+1,n]}^x}{\eta_i^x} \cdot  q^{2mx\xi_i^x +\xi_i^x \left( \eta_{[i,n]}^x + \sum_{z>x} 2\eta^z_{[i,n]}   \right) } ,
$$
where const is a constant that does not change under the dynamics. To see this, write
$$
\eta^y_{[0,i-1]} = m - \eta^y_{[i,n]}
$$
for all $y$ and $i$. Plugging this into the original definition of $D$ and setting const to be
$$
\text{const} = q^{ 2(L+1)m \vert \boldsymbol{\xi}\vert}
$$
shows the alternative form of $D$.

Similarly, define
$$
\tilde{D}(\xi,\eta) = \prod_x [\eta_0^x]_q^! \cdots [\eta_n^x]_q^! \prod_{i=1}^n \binomq{\eta_{[i,n]}^x - \xi_{[i+1,n]}^x}{\eta_i^x} \cdot  q^{\xi_i^x \left( \eta_{[i,n]}^x + \sum_{z>x} 2\eta^z_{[i,n]}   \right) }.
$$

The duality result is stated below as an intertwining; this intertwining is the definition of Markov duality. The $^*$ denotes the transposition of the matrix, and $\vec{z}\ ^{-1} $ denotes $ (z_1^{-1},\ldots,z_L^{-1})$.
\begin{theorem}\label{Only} For any $q$ and $\vec{z}$,
$$
{\mathcal{P}}^*(q,\vec{z} ) D = D \accentset{\circ}{\mathcal{P}}_{\text{rev}} (q^{-1} ,\vec{z}\ ^{-1} ) .
$$

\end{theorem}
\begin{proof}
We start with an intertwining of the $S$--matrices, which previously appeared but had not been applied to dualiy. Let $\Pi_{*}$ denote the matrix that reverses the order of the colors. In other words, $\Pi_{\mathcal{H}}$ and $\Pi_{\mathcal{V}}$ denote the permutation matrices with entries
$$
\Pi_{\mathcal{H}}( \alpha,\gamma) = 
\begin{cases} 
1, \text{ if } \alpha_i = \gamma_{n-i} \text{ for all } i\in \{0,1,\ldots,n\} \\
0, \text {else } 
\end{cases}
$$
and
$$
\Pi_{\mathcal{V}}( \beta,\delta) = 
\begin{cases} 
1, \text{ if } \beta_i = \delta_{n-i} \text{ for all } i\in \{0,1,\ldots,n\} \\
0, \text {else } 
\end{cases}.
$$
Proposition 3.7 of \cite{KuanCMP} implies that
\begin{equation}\label{I2}
\PiHV \circ S(q,z) \circ  \PiHV = S(q^{-1},z^{-1}).
\end{equation}
Alternatively, one can check directly when $l=m=1$, and then use the fusion relation in Theorem 3.4 of \cite{KuanCMP}. 

One consequence of \eqref{I2} is that
$$
(\Pi_{\mathcal{V}}\otimes \Pi_{\mathcal{H}}) \circ \bar{S}(q,z) \circ \PiHV = \bar{S}(q^{-1},z^{-1}).
$$
Letting $\Pi^{(L)}$ denote $\Pi_{\mathcal{H}} \otimes \Pi_{\mathcal{V}_{m_1}} \otimes \cdots \otimes \Pi_{\mathcal{V}_{m_L}}$ and $\bar{\Pi}^{(L)}$ denote $\Pi_{\mathcal{V}_{m_1}} \otimes \cdots \otimes \Pi_{\mathcal{V}_{m_L}}\otimes \Pi_{\mathcal{H}}$, the previous equation implies 
\begin{equation}
\mathcal{T}_L(q,\vec{z}) \circ \Pi^{(L)} = \bar{\Pi}^{(L)}\circ \mathcal{T}_L(q^{-1},\vec{z}\ ^{-1}).
\end{equation}
This implies (using the argument and results of section 4.4 of \cite{KuanCMP}) that
\begin{equation}\label{L3}
\accentset{\circ}{\mathcal{P}}_{\text{rev}}(q^{-1} ,\vec{z}\ ^{-1}) \circ \Pi = \Pi \circ {\mathcal{P}}_{\text{rev}}(q,\vec{z})  ,
\end{equation}
where $\Pi$ is defined on $\mathcal{V}_{m_1} \times \cdots \times \mathcal{V}_{m_L}$.

The previous duality result (Theorem 4.10\footnote{There is a typo in the paper \cite{KuanCMP}. Equations (16) and (17) of \cite{KuanCMP}, which define $\Pi D$,  cite \cite{KIMRN} but switch the $\eta$ and the $\xi$. This can also be corrected by reversing the direction of the jumps in $\eta$ and $\xi$  (which corresponds to switching $\eta$ and $\xi$), as noted by \cite{YL19}. In any case, we will not use the explicit expression of $\Pi D$ here. } of \cite{KuanCMP}) says that
\begin{equation}\label{4.10}
\mathcal{P}^*(q, \vec{z}) \circ D \Pi  = D \Pi  \circ \mathcal{P}_{\text{rev}}(q,\vec{z}).
\end{equation}
As noted in Remark 5 of \cite{KuanSF}, intertwinings of the form \eqref{L3} can be used to produce new dualities from old ones. Because $\Pi = \Pi^*$ and $\Pi^2=\mathrm{id}$,  plugging \eqref{L3} into \eqref{4.10} implies that
$$
{\mathcal{P}}^*(q,\vec{z} ) D \Pi = D \accentset{\circ}{\mathcal{P}}_{\text{rev}} (q^{-1} ,\vec{z}\ ^{-1} )\Pi .
$$
Multiplying by $\Pi$ on both sides shows the theorem.
\end{proof}

\begin{reemark}
If the lattice is the infinite line $\mathbb{Z}$, then the theorem can be written as
$$
{\mathcal{P}}^*(q,\vec{z} ) D = \mathfrak{d} D {\mathcal{P}}_{\text{rev}} (q^{-1} ,\vec{z}\ ^{-1} ) ,
$$
where $\mathfrak{d}$ is the diagonal matrix with entries $\mathfrak{d}(\bxi,\bxi) = q^{2 \vert \bxi\vert}$, where $\vert \bxi \vert$ denotes the number of particles in $\bxi$. This is because the particles entering from the right still contribute to the duality, even as the lattice size grows to infinity. See the first example in section 3 for an illustrative example.

If, on the other hand, the lattice is $\{\ldots,-2,-1\}$, then we have a duality between a stochastic vertex model with particles jumping to the left and entering at $-1$, and a stochastic vertex model with particles jumping to the right and exiting from $-1$. Similar types of boundary conditions have occurred in duality results; see \cite{GKRV}, \cite{CGRConsistent}, \cite{KuanSF}. The second example in section 3 also illustrates this type of boundary condition.
\end{reemark}

\begin{corollary}
Suppose that the lattice is the infinite line $\mathbb{Z}$, and all $m_x$ and $l$ equal $1$. Then
$$
\mathcal{P}^*(q ,\vec{z}) \tilde{D} = \tilde{D} \mathcal{P}_{\text{rev}}(q,\vec{z} ).
$$ 
\end{corollary}
\begin{proof}
Now assume the conditions in the second duality result. Let $S$ denote the shift operator defined by 
$$
S( \boldsymbol{\eta},\boldsymbol{\xi} ) = \begin{cases}
1, & \text{if } \eta^x = \xi^{x+1} \text{ for all } x, \\
0,& \text{else}.
\end{cases}
$$
We have that $S^{-1}=S^*$, and the translation invariance says that
$$
S^{-1}\mathcal{P}S  = \mathcal{P}, \quad S^{-1}\mathcal{P}_{\text{rev}}S = \mathcal{P}_{\text{rev}}.
$$
Additionally, from the expressions for $D$ and $\tilde{D}$ it is immediate that
$$
S^{-1}DS=D, \quad S^{-1} \tilde{D} S = \tilde{D}.
$$
Let $G$ be the diagonal matrix with entries
$$
G( \boldsymbol{\xi}, \boldsymbol{\xi}) = \prod_{x \in \mathbb{Z}} \prod_{i=1}^n q^{-2mx \xi^x_i} .
$$
From the definitions of $D$ and $\tilde{D}$, the matrix $G$ relates $D$ and $\tilde{D}$ via
$$
\tilde{D} = \text{const} \cdot GD.
$$

Once we show 
\begin{equation}\label{LF}
SGS^* \mathcal{P}^*(q^{}, \vec{z}\ ^{})G^{-1} = \mathcal{P}^*(q^{-1} ,\vec{z} \ ^{-1}),
\end{equation}
this proves the result. Indeed, assuming \eqref{LF},
\begin{align*}
\mathcal{P}^*(q^{}, \vec{z}\ ^{}) D &= D \accentset{\circ}{\mathcal{P}}_{\text{rev}}(q^{-1}, \vec{z}\ ^{-1}) \\
\Longrightarrow SGS^*\mathcal{P}^*(q^{}, \vec{z}\ ^{}) G^{-1}GD &= SGS^*D \accentset{\circ}{\mathcal{P}}_{\text{rev}}(q^{-1}, \vec{z}\ ^{-1}) \\
\Longrightarrow  \mathcal{P}^*(q^{-1}, \vec{z}\ ^{-1}) \tilde{D} &= S \tilde{D} S^* \accentset{\circ}{\mathcal{P}}_{\text{rev}}(q^{-1}, \vec{z}\ ^{-1}), 
\end{align*}
which finally yields $ \mathcal{P}^*(q^{-1}, \vec{z}\ ^{-1}) \tilde{D} =  \tilde{D}  \accentset{\circ}{\mathcal{P}}_{\text{rev}}(q^{-1}, \vec{z}\ ^{-1})$.

To show \eqref{LF}, first write it in the form
$$
 \mathcal{P}(q^{-1} ,\vec{z}\ ^{-1}) = G^{-1} \mathcal{P}(q^{}, \vec{z}) SGS^*.
$$
Set
$$
 b_1 = \frac{q^2(1-z)}{q^2-z}, \quad b_2 = \frac{1-z}{q^2-z}
$$
and note that $b_1/b_2 = q^2$, and the simultaneous inversions $q\mapsto q^{-1},z\mapsto z^{-1}$ switches $b_1$ and $b_2$. Every matrix entry of $\mathcal{P}(q^{-1}, \vec{z}\ ^{-1})$ is of the form
$$
b_2^{A_1} b_1^{A_2} ((1-b_1)(1-b_2))^{A_3}
$$
for some non--negative integers $A_1,A_2,A_3$. Proceed by induction on the value of $A_2$. 

Suppose $A_2=0$. Define a \textit{block} in the configuration $\bxi$ to be an interval $\{x,x+1,\ldots,y\}$ where $\xi^{x-1} \neq \xi^x = \xi^{x+1} =  \cdots = \xi^y \neq \xi^{y+1}.$ The condition that $A_2=0$ means that every particle either stayed put, or was part of a block that jumped one step to the right. The number of particles that stayed put is $A_1$, and the number of particles that were part of jumping blocks is $A_3$. Thus, the corresponding matrix entry of $G^{-1} \mathcal{P}(q^{}, \vec{z} ^{}) SGS^*$
is
$$
b_1^{A_1}  ((1-b_2)(1-b_1))^{A_3} \cdot { \left( \frac{b_2}{b_1}\right)}^{A_1}1^{A_3} , 
$$
which equals $b_2^{A_1} ((1-b_1)(1-b_2))^{A_3}$.

Now suppose we know \eqref{LF} for every matrix entry satisfying $A_2=r$. Let $\boldsymbol{\eta},\bxi$ be particle configurations such that the $(\boldsymbol{\eta},\bxi)$ entry of $\mathcal{P}(q^{-1}, \vec{z}\ ^{-1})$ equals
$$
b_2^{A_1} b_1^{r+1} ((1-b_1)(1-b_2))^{A_3}.
$$
Then there exists a $\boldsymbol{\zeta}$ such that the $(\boldsymbol{\zeta},\bxi)$ entry of $\mathcal{P}(q^{-1}, \vec{z}\ ^{-1})$ equals
$$
b_2^{C_1} b_1^r ((1-b_1)(1-b_2))^{C_3}
$$
for some $C_1,C_3$. The configuration $\boldsymbol{\zeta}$ can be constructed from $\boldsymbol{\eta}$ by taking a particle at lattice site $y$ and moving it to lattice site $x-1$ (where $x \leq y$), and $\{x,x+1,\ldots,y\}$ is a block in $\boldsymbol{\eta}$. Then 
$
C_1 - A_1 = y-x. 
$
The $(\boldsymbol{\zeta},\bxi)$ entry of $G^{-1} \mathcal{P}(q^{}, \vec{z}) SGS^* $ is, by definition, 
$$
[\mathcal{P}(q^{},\vec{z}) ](\boldsymbol{\zeta},\bxi) \frac{G(\boldsymbol{\xi}^+, \boldsymbol{\xi}^+)}{G(\boldsymbol{\zeta},\boldsymbol{\zeta})},
$$
where the superscript $^+$ denotes the particle configuration obtained by shifting every particle one lattice site to the right. By the induction hypothesis,
$$
[\mathcal{P}(q^{},\vec{z} ) ](\boldsymbol{\zeta},\bxi) \frac{G(\boldsymbol{\xi}^+, \boldsymbol{\xi}^+)}{G(\boldsymbol{\zeta},\boldsymbol{\zeta})}= [\mathcal{P}(q^{-1},\vec{z} \ ^{-1}) ](\boldsymbol{\zeta}, \bxi).  
$$
By construction, 
$$
G(\bzeta,\bzeta) = G(\beto,\beto) q^{2(y-x+1)} = G(\beto,\beto) \left( \frac{b_1}{b_2}\right)^{C_1-A_1+1}.
$$
So by translation invariance, 
$$
b_1^{C_1} b_2^r ((1-b_2)(1-b_1))^{C_3}\frac{G(\boldsymbol{\xi}^+, \boldsymbol{\xi}^+)}{G(\boldsymbol{\eta},\boldsymbol{\eta})} \left( \frac{b_2}{b_1}\right)^{C_1-A_1+1} = b_2^{C_1} b_1^r ((1-b_1)(1-b_2))^{C_3},
$$
which simplifies to 
$$
b_1^{A_1} b_2^{r+1} ((1-b_2)(1-b_1))^{C_3}\frac{G(\boldsymbol{\xi}^+, \boldsymbol{\xi}^+)}{G(\boldsymbol{\eta},\boldsymbol{\eta})} = b_2^{A_1} b_1^{r+1} ((1-b_1)(1-b_2))^{C_3},
$$
Multiplying both sides by $((1-b_1)(1-b_2))^{A_3-C_3}$, we arrive at
$$
[\mathcal{P}(q,\vec{z} ) ](\boldsymbol{\beto},\bxi) \frac{G(\boldsymbol{\xi}^+, \boldsymbol{\xi}^+)}{G(\boldsymbol{\beto},\boldsymbol{\beto})} = [\mathcal{P}(q^{-1} , \vec{z}\ ^{-1})]  (\beto, \bxi),
$$
as needed.

\end{proof}

\begin{reemark}
In the single--species ($n=1$) case of the corollary, the stochastic vertex model reduces to the stochastic six vertex model. After replacing $q$ by $q^{-1}$, the duality reduces to the duality in \cite{YL19}.
\end{reemark}

\begin{reemark}
In the degeneration of the stochastic six vertex model to ASEP, we recover a duality for ASEP. Heuristically, as $z\rightarrow 1$,
\begin{align*}
\mathcal{P}(q,z) &\approx S -  L_{\text{ASEP}({q^2,1})} S(z-1) + \mathcal{O}((z-1)^2) ,\\
\mathcal{P}(q^{-1},z^{-1}) &\approx S^* -  L_{\text{ASEP}({1,q^2})} S^*(z-1) + \mathcal{O}((z-1)^2) .
\end{align*}
Here, $L_{\text{ASEP}(l,r)}$ is the generator of ASEP with left jump rates $l$ and right jump rates $r$. The two dualities in this paper then become the Sch\"{u}tz ASEP duality \cite{Sch97} and the ASEP duality in Corollary 4.7(c) from \cite{KuanSF}. Note that the latter duality function had already occurred in \cite{Sch97}, using time reversal rather than space reversal. The relation \eqref{LF} then reduces to the intertwining in the proof of Theorem 4.6(b) of \cite{KuanSF}.
\end{reemark}

\begin{reemark}
As mentioned in the appendix of \cite{CorPetCMP}, the degeneration of the stochastic six vertex model to ASEP does not for higher spin $(m>1)$ models, because non--negativity fails to hold. Thus, the $m=1$ case can be regarded as a unique case; \eqref{LF} does not hold in general, for example.
\end{reemark}

\section{Examples}
When $l=1$, there is a simple expression for $S(z)$. For $0 \leq j \leq n$, let $\epsilon_j \in \HH$ denote the vector $(0,\ldots,0,1,0,\ldots,0)$, where the $1$ is located the $j$th index. For any $0 \leq i\leq j \leq n$, let $\alpha_{[i,j]} = \alpha_i  + \alpha_{i+1} + \cdots + \alpha_j$. From (25) of \cite{KuanCMP}\footnote{Note that we are using different notation than in that paper.}, 
\begin{equation}\label{Ex}
(q^{m+1}-z) S(z)_{\epsilon_j, \beta}^{\epsilon_k, \delta} = 1_{\{\epsilon_j + \beta = \epsilon_k + \delta\}} \times
\begin{cases}
q^{2\beta_{[k,n]}-m+1} ( 1- q^{-2\beta_k+m-1}z), &\text{ if } k=j ,\\
-q^{2\beta_{[k+1,n]}-m+1} (1-q^{2\beta_k}), &\text { if } k < j ,\\
-q^{2\beta_{[k+1,n]}}z (1-q^{2\beta_k}), & \text{ if } k > j .
\end{cases}
\end{equation}

Below are the four nontrivial vertex weights for $l=m=1$:

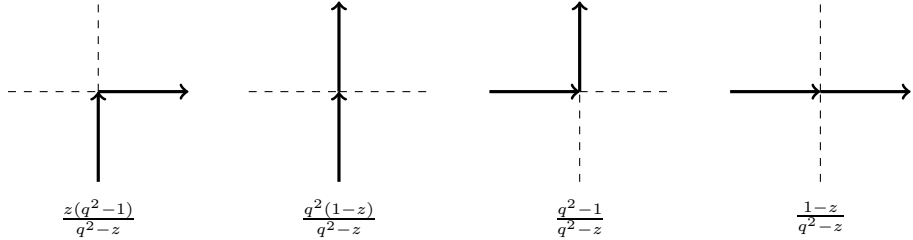
\begin{center}
\begin{tikzpicture}[scale=0.8]
\rotatebox{360}{
\draw [dashed](-1.5,0) -- (0,0) ;
\draw [->, very thick] (0,0) -- (1.5,0) ;
\draw [->, very thick] (0,-1.5) -- (0,0);
\draw [dashed] (0,0) -- (0,1.5);
\draw (0,-2.1) node  {$\frac{z(q^2-1)}{q^2-z}$} ;
}

\draw [dashed](2.5,0) -- (4,0) ;
\draw [dashed] (4,0) -- (5.5,0) ;
\draw [->, very thick] (4,-1.5) -- (4,0);
\draw [->, very thick] (4,0) -- (4,1.5);
\draw (4,-2.1) node  {$\frac{q^2(1-z)}{q^2-z}$} ;

\draw [->, very thick](6.5,0) -- (8,0) ;
\draw [dashed] (8,0) -- (9.5,0) ;
\draw [dashed] (8,-1.5) -- (8,0);
\draw [->, very thick] (8,0) -- (8,1.5);
\draw (8,-2.1) node  {$\frac{q^2-1}{q^2-z}$} ;

\draw [->, very thick](10.5,0) -- (12,0) ;
\draw [->, very thick] (12,0) -- (13.5,0) ;
\draw [dashed] (12,-1.5) -- (12,0);
\draw [dashed] (12,0) -- (12,1.5);
\draw (12,-2.1) node  {$\frac{1-z}{q^2-z}$} ;
\end{tikzpicture}
\end{center}
In the image below, {the  image shows $\mathcal{T}_L^{\text{rev}}$, which has weight  $(q^2-z)^{-3} (-z(1-q^2)) \cdot (1-z) \cdot (-(1-q^2) )$.}

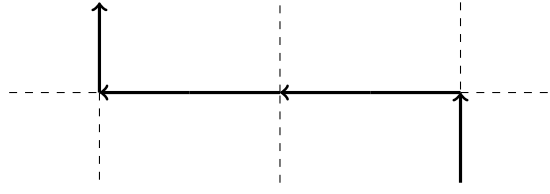
\begin{center}

\begin{tikzpicture}[scale=0.8]
\draw [dashed](-1.5,0) -- (0,0) ;
\draw [<-, very thick] (0,0) -- (1.5,0) ;
\draw [very thick](1.5,0) -- (3,0) ;
\draw [<-, very thick](3,0) --(4.5,0);
\draw [very thick](4.5,0) -- (6,0) ;
\draw [dashed](6,0) -- (7.5,0) ;
\draw [<-, very thick] (0,1.5) -- (0,0);
\draw [dashed] (0,0) -- (0,-1.5);
\draw [dashed] (3,-1.5) -- (3,0);
\draw [dashed] (3,0) -- (3,1.5);
\draw [dashed] (6,1.5) -- (6,0);
\draw [<-,very thick] (6,0) -- (6,-1.5);
\end{tikzpicture}
\end{center}

Below are the eight nontrivial vertex weights when $l=1,m=2$. 

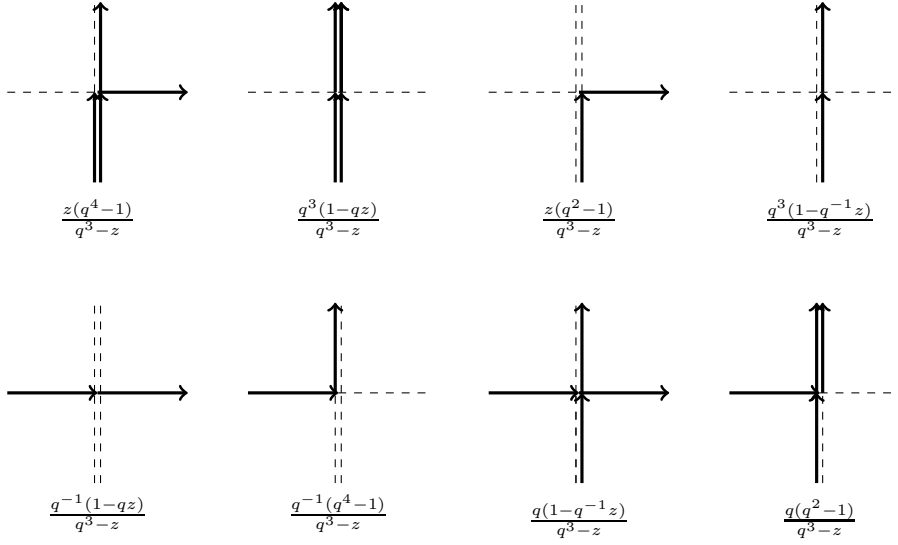
\begin{center}
\begin{tikzpicture}[scale=0.8]
\draw [dashed](-1.5,0) -- (0,0) ;
\draw [->, very thick] (0,0) -- (1.5,0) ;
\draw [->, very thick] (-0.05,-1.5) -- (-0.05,0);
\draw [->, very thick] (0.05,-1.5) -- (0.05,0);
\draw [dashed] (-0.05,0) -- (-0.05,1.5);
\draw [->, very thick] (0.05,0) -- (0.05,1.5);
\draw (0,-2.1) node  {$\frac{z(q^4-1)}{q^3-z}$} ;
\draw [dashed](2.5,0) -- (4,0) ;
\draw [dashed] (4,0) -- (5.5,0) ;
\draw [->, very thick] (3.95,-1.5) -- (3.95,0);
\draw [->, very thick] (4.05,-1.5) -- (4.05,0);
\draw [->, very thick] (3.95,0) -- (3.95,1.5);
\draw [->, very thick] (4.05,0) -- (4.05,1.5);
\draw (4,-2.1) node  {$\frac{ q^3(1-qz) }{q^3-z}$} ;
\draw [dashed](6.5,0) -- (8,0) ;
\draw [->, very thick] (8,0) -- (9.5,0) ;
\draw [dashed] (7.95,-1.5) -- (7.95,0);
\draw [->, very thick] (8.05,-1.5) -- (8.05,0);
\draw [dashed] (7.95,0) -- (7.95,1.5);
\draw [dashed] (8.05,0) -- (8.05,1.5);
\draw (8,-2.1) node  {$\frac{ z(q^2-1) }{q^3-z}$} ;
\draw [dashed] (10.5,0) -- (12,0) ;
\draw [dashed] (12,0) -- (13.5,0) ;
\draw [dashed] (11.95,-1.5) -- (11.95,0);
\draw [->, very thick] (12.05,-1.5) -- (12.05,0);
\draw [dashed] (11.95,0) -- (11.95,1.5);
\draw [->, very thick] (12.05,0) -- (12.05,1.5);
\draw (12,-2.1) node  {$\frac{ q^3(1-q^{-1}z) }{q^3-z}$} ;
\draw [->, very thick] (0,-5) -- (1.5,-5) ;
\draw [<-, very thick] (0,-5) -- (-1.5,-5) ;
\draw [dashed] (-0.05,-6.5) -- (-0.05,-5);
\draw [dashed] (0.05,-6.5) -- (0.05,-5);
\draw [dashed] (-0.05,-5) -- (-0.05,-3.5);
\draw [dashed] (0.05,-5) -- (0.05,-3.5);
\draw [->, very thick] (4.05,0) -- (4.05,1.5);
\draw (0,-7) node {$\frac{q^{-1}(1-  qz)}{q^3-z}$};
\draw [dashed] (4,-5) -- (5.5,-5) ;
\draw [<-, very thick] (4,-5) -- (2.5,-5) ;
\draw [dashed] (3.95,-6.5) -- (3.95,-5);
\draw [dashed] (4.05,-6.5) -- (4.05,-5);
\draw [->, very thick] (3.95,-5) -- (3.95,-3.5);
\draw [dashed] (4.05,-5) -- (4.05,-3.5);
\draw (4,-7) node {$\frac{q^{-1}(q^4 -1 )}{q^3-z}$};
\draw [<-, very thick](8,-5) -- (6.5,-5);
\draw [dashed](7.95,-6.5) -- (7.95,-5) ;
\draw [->, very thick] (8,-5) -- (9.5,-5) ;
\draw [dashed] (8-0.05,-5-1.5) -- (8-0.05,0-5);
\draw [->, very thick] (8+0.05,-5-1.5) -- (8+0.05,0-5);
\draw [dashed] (8-0.05,0-5) -- (8-0.05,1.5-5);
\draw [->, very thick] (8+0.05,0-5) -- (8+0.05,1.5-5);
\draw (8,-7.1) node  {$\frac{ q(1-q^{-1}z)}{q^3-z}$} ;
\draw [->, very thick](8+2.5,-5) -- (8+4,-5) ;
\draw [dashed] (8+4,-5) -- (8+5.5,-5) ;
\draw [->, very thick] (8+3.95,-5-1.5) -- (8+3.95,-5);
\draw [dashed] (8+4.05,-5-1.5) -- (8+4.05,-5);
\draw [->,very thick] (8+3.95,-5) -- (8+3.95,-5+1.5);
\draw [->, very thick] (8+4.05,-5) -- (8+4.05,-5+1.5);
\draw (8+4,-7.1) node  {$\frac{ q(q^2-1) }{q^3-z}$} ;
\end{tikzpicture}
\end{center}

We work out some examples:
\begin{itemize}
\item
Let us verify that on the infinite line, 
$$
[\mathcal{P}^*(q ,\vec{z} ) D] (\bxi, \boldsymbol{\eta}) = [D \accentset{\circ}{\mathcal{P}}_{\text{rev}}(q^{-1} ,\vec{z}\ ^{-1} )](\bxi, \boldsymbol{\eta}) 
$$
when $\bxi$ consists solely of one particle at lattice site $0$ and $\boldsymbol{\eta}$ consists of $k$ particles at lattice site $0$, and $m_0=m$. The left--hand--side simplifies as 
$$
[\mathcal{P}^*(q^{}  ,\vec{z}\ ^{} ) ](\bxi,\bxi) D(\bxi,\boldsymbol{\eta}) =  \frac{q^{m+1}(1-q^{-(m+1)}z)}{q^{m+1}-z}  \cdot \frac{q^{k} - q^{-k}}{q-q^{-1}} q^{k}.
$$
The right--hand--side has two terms, one corresponding to when all particles in $\boldsymbol{\eta}$ stay still, and the other corresponding to when one particle jumps:
$$
[D \accentset{\circ}{\mathcal{P}}_{\text{rev}}(q^{-1} ,\vec{z}\ ^{-1} )](\bxi, \boldsymbol{\eta}) 
= 
\left(\frac{q^{-(m+1)}(1 - q^{m-2k+1}z^{-1}) }{q^{-(m+1)}-z^{-1}} \cdot \frac{q^k - q^{-k}}{q-q^{-1}} q^{k}
+
\frac{z^{-1}(q^{-2k}-1)}{q^{-(m+1)}-z^{-1}}  \cdot \frac{q^{k-1} - q^{-k+1}}{q-q^{-1}} q^{k-1}\right)q^2.
$$
The factor of $q^2$ comes from the particle entering from the right. From a direct calculation, both sides equal
\begin{equation*}
\frac{q^{k+1}\left(q^k - q^{-k} \right)\left(-q^{m}+q z\right)}{\left(q-q^{-1}\right)\left(q^{1+m}-z\right)}.
\end{equation*}
\item
Let us verify that on $\{\ldots,-3,-2,-1\}$,
$$
[\mathcal{P}^*(q ,\vec{z} ) D] (\bxi, \boldsymbol{\eta}) = [D \accentset{\circ}{\mathcal{P}}_{\text{rev}}(q^{-1} ,\vec{z}\ ^{-1} )](\bxi, \boldsymbol{\eta}) 
$$
where $\bxi$ consists solely of one particle at lattice site $-r$ and $\beto$ is empty. The left--hand--side is
$$
\frac{ z(q^2-1)}{(q^{m_{-r}+1}-z)} \cdot \frac{q^{-m_{-r+1}+1} - z}{q^{m_{-r+1}+1}-z} \cdots \cdot \frac{q^{-m_{-1}+1}-z}{q^{m_{-1}+1}-z},
$$
where the right--hand--side is
\begin{align*}
&\frac{q^{-m_{-1}+1}-z}{q^{m_{-1}+1}-z} \cdot \cdots \frac{q^{-m_{-r+1}+1} - z}{q^{m_{-r+1}+1}-z} \cdot \frac{q^{-m_r+1}(q^{2m_r}-1)}{q^{m_r+1}-z}\Big|_{z\mapsto z^{-1},q\mapsto q^{-1}} \cdot D(\bxi,\bxi) \\
&= \frac{q^{2m_{-1}}(q^{-m_{-1}+1}-z)}{q^{m_{-1}+1}-z} \cdot \cdots \frac{q^{2m_{-r+1}}(q^{-m_{-r+1}+1} - z)}{q^{m_{-r+1}+1}-z} \cdot \frac{zq^{m_{-r}-1}\cdot q^{-m_{-r}+1}(q^{2m_{-r}}-1)}{q^{m_{-r}+1}-z} \\
& \quad \quad \times \frac{1}{[m_r]_q} q^{-(m_{-r}-1) -\sum_{s=-r+1}^{-1} 2m_s} ,\\ 
\end{align*}
which equals the left--hand--side.

\end{itemize}

\bibliographystyle{alpha}
\bibliography{ShortNotev2}

\end{document}